\documentclass[11pt, letterpaper, oneside]{amsart}

\usepackage{latexsym, amsmath, amssymb, amsfonts, amscd,bm}
\usepackage{amsthm}
\usepackage{t1enc}
\usepackage[mathscr]{eucal}
\usepackage{indentfirst}
\usepackage{graphicx, pb-diagram}
\usepackage{fancyhdr}
\usepackage{fancybox}
\usepackage{enumerate}
\usepackage{color}
\usepackage[all]{xy}

\usepackage{url}

\theoremstyle{plain}
\newtheorem{thm}{Theorem}[section]

\newtheorem{prop}[thm]{Proposition}

\newtheorem{lemma}[thm]{Lemma}
\newtheorem{cor}[thm]{Corollary}

\theoremstyle{definition}
\newtheorem{defi}[thm]{Definition}

\theoremstyle{remark}
\newtheorem{remark}[thm]{Remark}

\newtheorem*{ack}{Acknowledgements}

\newcommand{\ZZ}{\ensuremath{\mathbb Z}}

\newcommand{\RR}{\ensuremath{\mathbb R}}

\renewcommand{\graph}{\mathrm{graph}}

\newcommand{\dom}{\mathrm{dom}}

\newcommand{\Kr}{\mathrm{Kr}}

\newcommand{\pd}[1]{{\partial_{#1}}} 
\newcommand{\li}{\ensuremath{L_{\infty}}}

\definecolor{forest}{rgb}{0,0.5,0}

\begin{document}

\title[Deformations of coisotropic submanifolds]{Deformations of coisotropic submanifolds for fibrewise entire Poisson structures}

\author{Florian Sch\"atz}
\address{Department of Mathematics, Utrecht University, 3508 TA Utrecht, The Netherlands}
\email{florian.schaetz@gmail.com}

\author{Marco Zambon}
\address{Universidad Aut\'onoma de Madrid (Departamento de Matem\'aticas) and ICMAT(CSIC-UAM-UC3M-UCM),
Campus de Cantoblanco,
28049 - Madrid, Spain}
\email{marco.zambon@uam.es, marco.zambon@icmat.es}

\date{\today}
\thanks{2000 Mathematics Subject Classification:   primary,  secondary   .}

\begin{abstract}
We show that deformations of a coisotropic submanifold
inside a fibrewise entire Poisson manifold are controlled by the
$L_\infty$-algebra introduced by Oh-Park (for symplectic manifolds) and Cattaneo-Felder.
In the symplectic case, we recover results previously obtained by Oh-Park.
Moreover we consider the extended deformation problem and prove its obstructedness.
\end{abstract}

\thanks{
}

\maketitle
\setcounter{tocdepth}{1} 
\tableofcontents

\section*{Introduction}\label{intro}

 We consider deformations of coisotropic submanifolds inside
 a fixed Poisson manifold $(M,\pi)$, with $\pi$ a fibrewise entire Poisson structure, see Definition \ref{defi:fibrewise_entire}.

We build on work of Oh and Park \cite{OP}, who realized that deformations of a coisotropic submanifold inside a symplectic manifold are governed by an $L_{\infty}[1]$-algebra.
 The construction of the $L_{\infty}[1]$-algebra structure was extended
\cite{CaFeCo2} to arbitrary coisotropic submanifolds of Poisson manifolds by Cattaneo and Felder. The $L_{\infty}[1]$-algebra depends only on the $\infty$-jet of $\pi$ along the coisotropic submanifold $C$, so it has too little information to codify $\pi$ near $C$ in general. In particular it does not encode the coisotropic submanifolds of $(M,\pi)$ nearby $C$, see \cite[Ex. 3.2 in \S 4.3]{FloDiss} for an
example of this.

In this note we show that if the identification between a tubular neighbourhood of $C$ in $M$ and a neighbourhood in its normal bundle $NC$ is chosen so that $\pi$ corresponds to a  \emph{fibrewise entire} bivector field on $NC$, the $L_{\infty}[1]$-algebra structure encodes coisotropic submanifolds of $(M,\pi)$ nearby $C$; see \S\ref{sec:def}.
For instance, such an identification exists for coisotropic submanifolds of symplectic manifolds; see {\S\ref{sec:sym}}. Further, we show that the problem of deforming simultaneously
the Poisson structure $\pi$ and the coisotropic submanifold $C$ is
formally obstructed; see \S\ref{sec:sim}.

\begin{ack}
M.Z. thanks Daniel Peralta-Salas and { Alberto Mart\'in Zamora
for their help on analytic issues, and Andrew Lewis for interesting discussions.}
F.S. thanks Alberto Cattaneo and Camille Laurent-Gengoux for stimulating discussions.
M.Z. was partially supported by  projects PTDC/MAT/098770/2008 and PTDC/MAT/099880/2008 (Portugal) and by  projects MICINN RYC-2009-04065,  MTM2009-08166-E, MTM2011-22612, ICMAT Severo Ochoa project SEV-2011-0087 (Spain). F.S. was supported by ERC Starting Grant no. 279729.
\end{ack}

\section{Deformations in the fibrewise entire case}\label{sec:def}

We introduce the notion of a fibrewise entire (multi)vector field on a vector bundle $E\to C$.
Then we show that whenever $E$ is equipped with a fibrewise entire Poisson structure such that $C$ is coisotropic, the $L_{\infty}[1]$-algebra structure associated to $C$ encodes coisotropic submanifolds sufficiently close to $C$ (in the $\mathcal{C}^1$-topology).

\subsection{Fibrewise entire multivector fields}

Let $E\to C$ be a vector bundle throughout this subsection.

\begin{defi}
The set $\mathcal{C}^{\omega}(E)$ of {\em locally defined fibrewise entire functions}
on $E$ contains
those smooth functions $f: U\to \mathbb{R}$ which are defined on some tubular neighborhood $U$ of $C$ in $E$
and whose restriction to each fibre $U_x = U \cap E_x$ is given by a convergent power series.

Given $f\in \mathcal{C}^{\omega}(E)$, we denote by $\dom(f)$ the tubular neighbourhood on which $f$ is defined.
\end{defi}

\begin{remark}
\hspace{0cm}
\begin{enumerate}
\item To be more accurate, ``entire'' should read ``real entire'' in the above definition. We will usually also drop the term ``locally defined'' and simply
refer to $\mathcal{C}^{\omega}(E)$ as the {\em fibrewise entire functions} on $E$.
\item The set $\mathcal{C}^{\omega}(E)$ forms an algebra under the usual multiplication of functions.
\end{enumerate}
\end{remark}

\begin{defi}
The set $\chi_{\omega}(E)$ of {\em fibrewise entire vector fields} contains those smooth vector fields $X$ which are defined on some tubular neighborhood $U$ of $C$ in $E$
and whose actions on (locally defined) functions sends fibrewise polynomial functions to $\mathcal{C}^{\omega}(E)$.

Given $X \in \chi_{\omega}(E)$, we denote by $\dom(X)$ the tubular neighborhood on which $X$ is defined.
\end{defi}

We state two other descriptions of $\chi_\omega(E)$:

\begin{lemma}
Given $X$ a smooth vector field defined on some tubular neighborhood $U$ of $C$ in $E$,
the following assertions are equivalent:
\begin{enumerate}
 \item $X$ is fibrewise entire.
 \item If $(x_i)_{i=1}^m$ are local coordinates on $W\subset C$ and $(y_j)_{j=1}^n$ are fibre coordinates on $E|_W$, then
$X$ reads
$$ \sum_{i=1}^m h_i(x,y)\partial_{x_i} + \sum_{j=1}^n g_j(x,y)\partial_{y_j} $$
on $E|_W$ with $(h_i)_{i=1}^m$ and $(g_j)_{j=1}^n$ fibrewise entire functions.
\end{enumerate}

\end{lemma}

\begin{proof}
Since being fibrewise entire is a local property with respect to the base manifold $C$,
the equivalence of (1) and (2) can be easily checked in coordinates.
\end{proof}

\begin{remark}
The requirement that a vector field $X$ be fibrewise entire seems {\em not}
to be equivalent to the requirement that the action of $X$ on smooth functions preserves the subalgebra $\mathcal{C}^{\omega}(E)$.
The difference between these two requirements should already be visible in the simplest case:
there should be a smooth function $$f: \mathbb{R}^2 \to \mathbb{R}$$ in two variables $x$ and $y$
such that for each fixed $x$, the function $y \mapsto f(x,y)$ is globally analytic, i.e. its Taylor expansion around $0$
converges to $f(x,y)$, while its partial derivative $\frac{\partial f}{\partial x}$ is {\em not} globally analytic.
Note that $\frac{\partial f}{\partial x}$ is $X(f)$ for the fibrewise entire vector field $\partial_x$.
\end{remark}

\begin{defi}\label{defi:fibrewise_entire}
The set $\chi^{\bullet}_{\omega}(E)$ of {\em fibrewise entire multivector fields} contains those smooth multivector fields $Z$, defined on some tubular neighborhood $U$ of $C$ in $E$, whose action
on (locally defined) functions sends fibrewise polynomial functions to $\mathcal{C}^{\omega}(E)$.

Given $Z \in \chi^{\bullet}_{\omega}(E)$, we denote by $\dom(Z)$ the tubular neighborhood on which $Z$ is defined.
\end{defi}

As for vector fields, one can give different, but equivalent, characterizations of fibrewise entire multivector fields:

\begin{lemma}
Given $Z$ a smooth $k$-vector field defined on some tubular neighborhood $U$ of $C$ in $E$,
the following assertions are equivalent:
\begin{enumerate}
 \item $Z$ is fibrewise entire.
 \item If $(x_i)_{i=1}^m$ are local coordinates on $W\subset C$ and $(y_j)_{j=1}^n$ are fibre coordinates on $E|_W$, then
$Z$ reads
$$ \sum_{r+s=k}\sum_{i_1,\dots,i_r,j_1,\dots,j_s} h_{i_1\dots i_r j_1\dots j_s}(x,y)\partial_{x_{i_1}}\wedge \cdots \wedge \partial_{x_{i_r}} \wedge
\partial_{y_{j_1}}\wedge \cdots \partial_{y_{j_s}} $$
on $E|_W$ with $(h_{i_1\dots i_r j_1\dots j_s})$ in $\mathcal{C}^{\omega}(E)$.
\item $Z$ can be written as the sum of wedge products of elements of $\chi_{\omega}(E)$.
\end{enumerate}
\end{lemma}

\begin{remark}
Notice that in particular a Poisson bivector field $\pi$ lies in $\chi^{2}_{\omega}(E)$
iff the Poisson bracket $\{f,g\}=\pi(df,dg)$ of fibrewise polynomial functions lies in $\mathcal{C}^{\omega}(E)$.
\end{remark}

\subsection{Deformation of coisotropic submanifolds}

For more background information and examples, the reader is advised to consult \cite{alancoiso}.

\begin{defi}
A submanifold $C$ of a Poisson manifold $(M,\pi)$ is {\em coisotropic} if the restriction of the bundle map
$$ \pi^{\sharp}: T^*M \to TM, \quad \xi \mapsto \pi(\xi,\cdot)$$
to the conormal bundle $N^*C:=TC^{\circ}$   takes values in $TC$.
\end{defi}

\begin{defi}  An \emph{$L_\infty[1]$-algebra} is a $\ZZ$-graded vector space $W$, equipped with a collection of graded symmetric brackets $(\lambda_k\colon W^{\otimes k} \longrightarrow W)_{k\ge1}$ of degree $1$ which satisfy a collection of quadratic relations \cite{LadaStasheff} called higher Jacobi identities.

The {\em Maurer-Cartan series} of an element $\alpha$ of $W$ of degree $0$ is the infinite sum
$$ \mathrm{MC}(\alpha) := \sum_{k\ge 1} \frac{1}{k!}\lambda_k(\alpha^{\otimes k}).$$
\end{defi}

\begin{remark}\label{remark:L_infty[1]}
\hspace{0cm}
\begin{enumerate}
\item \label{one} Let $E\to C$ be a vector bundle. We denote by $P: \chi^{\bullet}(E)\to \Gamma(\wedge E)$ the map given by restriction to $C$, composed with the projection $\wedge (TE)|_C \to \wedge E$ induced by the splitting $(TE)|_C=E\oplus TC$.
The zero section $C$ of $E$ is coisotropic if and only if the image of $\pi$ under $P$ is zero.
 \item Suppose that $E$ is equipped with a Poisson structure $\pi$ with respect to which
$C$ is coisotropic. As shown in \cite{OP} and \cite{CaFeCo2}, the space $\Gamma(\wedge E)[1]$ is equipped with a canonical $L_\infty[1]$-algebra structure.
We denote the structure maps of this $L_\infty[1]$-algebra by
$$ \lambda_k\colon \Gamma(\wedge E)[1]^{\otimes k} \to \Gamma(\wedge E)[1].$$
Evaluating $\lambda_k$ on $\alpha^{\otimes k}$ for $\alpha \in \Gamma(E)$ yields
$$ \lambda_k(\alpha^{\otimes k}) := P\big([[\dots[\pi,\alpha],\alpha]\dots],\alpha] \big),$$
where $\alpha$ is interpreted as a fibrewise constant vertical vector field on $E$. Hence the Maurer-Cartan series of $\alpha$ reads
$MC(\alpha)=P(e^{[\cdot,\alpha]}\pi)$.
\end{enumerate}
\end{remark}
The aim of this subsection is  to prove:

\begin{thm}\label{main}
Let $E\to C$ be a vector bundle and $\pi$ a fibrewise entire Poisson structure which is defined on a tubular neighborhood $U=\dom(\pi)$ of $C$ in $E$.
Suppose that $C$ is coisotropic with respect to $\pi$.

Given a section $\alpha$ of $E$ for which $\graph(-\alpha)$ is contained in $U$, the Maurer-Cartan series $\mathrm{MC}(\alpha)$ converges
and its limit is $P((\phi^{\alpha})_*\pi)$, where $\phi^\alpha$ the time-$1$-flow of $\alpha$.

Hence, for such $\alpha$, the following two statements are equivalent:
\begin{enumerate}
 \item The graph of $-\alpha$ is a coisotropic submanifold of $(U,\pi)$.
 \item The Maurer-Cartan series $MC(\alpha)$ of $\alpha$ converges to zero.
\end{enumerate}
\end{thm}

\begin{proof}
We know from Proposition \ref{lem:phipi} below that $e^{[\cdot,\alpha]}\pi$ restricted to $C$ converges and the limit is $((\phi^\alpha)_*\pi)|_C$.
This implies that $\mathrm{MC}(\alpha)$ converges as well and the limit is $P((\phi^\alpha)_*\pi)$.

To prove the second part of the theorem, recall that $\phi^\alpha \colon E \to E$ is just translation by $\alpha$. Clearly $\graph(-\alpha)$ is coisotropic (for $\pi$) if{f}  $\phi^{\alpha}(\graph(-\alpha))=C$ is coisotropic for $(\phi^\alpha)_*\pi$.
The latter conditions is equivalent to $P((\phi^\alpha)_*\pi)=0$ by Remark \ref{remark:L_infty[1]} (\ref{one}).
\end{proof}

The equation $\mathrm{MC}(\alpha)=0$ is the Maurer-Cartan equation associated to the $L_\infty[1]$-algebra structure on $\Gamma(\wedge E)[1]$,
see Remark \ref{remark:L_infty[1]}.
 Theorem \ref{main} asserts that for $\alpha$ sufficiently $\mathcal{C}^0$-small,
the Maurer-Cartan equation is well-defined and its solutions correspond to coisotropic submanifolds which are sufficiently close to $C$ in the $\mathcal{C}^{1}$-topology.

\begin{remark}
\hspace{0cm}
\begin{enumerate}
\item The convergence of $\mathrm{MC}(\alpha)$ is meant pointwise, i.e. if we consider the sequence of sections $\beta_n \in \Gamma(E)$ defined by
$$ \beta_n:= \sum_{k=1}^n \frac{1}{k!} \lambda_k(\alpha^{\otimes k}),$$
convergence of $\mathrm{MC}(\alpha)$ means that for each $x\in C$, the sequence $(\beta_n(x))_n \subset E_x$ converges.

\item In \cite[\S5.1]{YaelZ}  the above Theorem is claimed for fibrewise polynomial Poisson bivector fields.
\end{enumerate}
\end{remark}

We devote the rest of this subsection to  Proposition \ref{lem:phipi} and its proof. We first need:
\begin{lemma}\label{lem:help}
\begin{enumerate}
 \item[(i)]
Let $\alpha \in \RR^n$, and denote by $\phi^\alpha \colon \RR^n \to \RR^n$ the time-$1$ flow of $\alpha$, i.e. $\phi^\alpha$ is just translation by $\alpha$.
Let $U$ be a neighborhood of the origin in $\RR^n$,
 such that $\alpha\in U$.

Then, for any entire function $f$ defined on $U$, the series
\begin{equation*}
(e^\alpha f)(0) := \sum_{k\ge 0} \frac{1}{k!} ([\alpha,[\dots[\alpha,[\alpha,f]]\dots]])(0)
\end{equation*}
converges to $((\phi^{\alpha})^*f)(0)$.

\item[(ii)] Let $E\to C'$ be a trivial vector bundle, $(x_i)_{i=1}^m$ coordinates on $C'$ and $(y_j)_{j=1}^n$ fibrewise linear functions on $E$, so that $(x_i,y_j)$ is a coordinate system on $E$.
View $\alpha\in \Gamma(E)$ as a vertical vector field on $E$ and denote by $\phi^\alpha$ its time-$1$-flow.
 Then the series
\begin{equation*}
e^{[\cdot,\alpha]}\pd{x_i} := \sum_{k\ge 0}\frac{1}{n!} [[\dots [[\pd{x_i},\alpha],\alpha]\dots ],\alpha] \quad \textrm{and} \quad e^{[\cdot,\alpha]}\pd{y_j}
\end{equation*}
are finite sums and equal $(\phi^\alpha)_*\pd{x_i}$ and $(\phi^{\alpha})_*\pd{y_j}$, respectively, at all points of $E$.
\end{enumerate}
\end{lemma}
\begin{proof}
(i) Denote by $(y_j)_{j=1}^n$ the canonical coordinates on $\RR^n$. We may assume that $\alpha=\pd{y_1}$.  We have

$$  (e^\alpha f)(0) = \sum_{n=0}^{\infty}\frac{1}{n!} ((\pd{y_1})^nf)|_0 = f(0+e_1) = ((\phi^{\alpha})^*f)(0). $$

where $e_1$ is first basis vector and we used Taylor's formula in the second equality.

(ii) Let us write in coordinates $\alpha=\sum_{j=1}^n f_j(x)\pd{y_j}$. Then
$[\pd{x_i},\alpha]=\sum_{j=1}^n \frac{\partial f_j}{\partial x_i}(x)\pd{y_j}$, so in particular the series $e^{[\cdot,\alpha]}\pd{x_i}$ is finite,
more precisely
$$e^{[\cdot,\alpha]}\pd{x_i}=\pd{x_i}+[\pd{x_i},\alpha] = \pd{x_i} + \sum_{j=1}^n \frac{\partial f_j}{\partial x_i}(x)\pd{y_j}.$$
On the other hand,
$$(\phi^\alpha)_*\pd{x_i}=\pd{x_i}+\sum_{j=1}^n \frac{\partial f_j}{\partial x_i}(x)\pd{y_j}.$$
The same reasoning applies to $e^{[\cdot,\alpha]}\pd{y_j}$.
\end{proof}

\begin{prop}\label{lem:phipi}
Let $E\to C$ be  vector bundle, and
 let $\pi\in \chi^{2}_{\omega}(E)$ be a fibrewise entire bivector field, defined on a tubular neighborhood $U=\dom(\pi)$ of $C$.
Let $\alpha \in \Gamma(E)$ such that $\graph(-\alpha)$ is contained in $U$.

Then the series
\begin{equation*}
e^{[\cdot,\alpha]}\pi := \sum_{k\ge 0} \frac{1}{k!} [[\dots [[\pi,\alpha],\alpha]\dots ],\alpha]
\end{equation*}
converges pointwise
on $C$ towards $((\phi^\alpha)_*\pi)|_C$, where $\phi^\alpha$ is the time-$1$-flow of $\alpha$.

\end{prop}
\begin{proof}
Choose local coordinates $(x_i)_{i=1}^m$ on $W\subset C$ and fibrewise linear functions $(y_j)_{j=1}^n$ on $E|_W$, so that $(x_i,y_j)$ is a coordinate system on $E|_W$. On the open subset $U|_W$ of $E$, write $\pi$ in these coordinates:
$$\pi= E + F + G = \sum_{i,i'}g_{ii'}(x,y)\pd{x_i}\wedge \pd{x_{i'}}+\sum_{i,j}h_{ij}(x,y)\pd{x_i}\wedge \pd{y_j}+\sum_{j,j'}k_{jj'}(x,y)\pd{y_j}\wedge \pd{y_{j'}}.$$
Since $\pi$ is fibrewise entire, the functions $g_{ii'}(x,y),h_{ij}(x,y),k_{jj'}(x,y)$ are in $\mathcal{C}^{\omega}(U|_W)$.

Hence, in a neighborhood of $W$ in $U|_W$, the pushforward bivector field $(\phi^{\alpha})_*\pi$ is equal to 
\begin{eqnarray*}
\sum_{i,i'}(\phi^{-1})^*g_{ii'}\cdot\phi_*\pd{x_i}\wedge \phi_*\pd{x_{i'}}+\sum_{i,j}(\phi^{-1})^*h_{ij}\cdot\phi_*\pd{x_i}\wedge \phi_*\pd{y_j} +\sum_{j,j'}(\phi^{-1})^*k_{jj'}\cdot\phi_*\pd{y_j}\wedge \phi_*\pd{y_{j'}},
\end{eqnarray*}
where we write $\phi:=\phi^{\alpha}$.
Restriction to $C$ and Lemma \ref{lem:help} yield, for all $x\in W$,
\begin{align*}
\big((\phi^{\alpha})_*\pi\big)|_{(x,0)} =  E_x' + F_x' + G_x' = \sum_{i,i'}(e^{-{\alpha}}g_{ii'})(x,0)\cdot e^{[\cdot,{\alpha}]}\pd{x_i}|_{(x,0)}\wedge e^{[\cdot,{\alpha}]}\pd{x_{i'}}|_{(x,0)} + \cdots.
\end{align*}
We claim that
$$ E_x' = (e^{[\cdot,\alpha]}E)|_{(x,0)}, \quad F_x' = (e^{[\cdot,\alpha]}F)|_{(x,0)} \quad \textrm{and} \quad G_x' = (e^{[\cdot,\alpha]}G)|_{(x,0)},$$
and hence $\big((\phi^\alpha)_*\pi\big)|_{(x,0)} = \big(e^{[\cdot,\alpha]} \pi \big)|_{(x,0)}$.

Observe that
the power series $e^{[\cdot,\alpha]}(g\pd{x_i})$ can be written as the Cauchy-product of the power series
$e^{-\alpha}g$ and the finite sum $e^{[\cdot,\alpha]}\pd{x_i}$. Hence $\big(e^{[\cdot,\alpha]}(g\pd{x_i})\big)|_{(x,0)}$ converges
to $\big(e^{-\alpha}g\big)|_{(x,0)} \cdot \big(e^{[\cdot,\alpha]}\pd{x_i}\big)|_{(x,0)}$. The analogous statement for
$g\pd{y_j}$ holds as well.
The claims about $E'_x$, $F'_x$ and $G'_x$ immediately follow.

\end{proof}

\section{The symplectic case}\label{sec:sym}

Throughout this section, $C$ is a coisotropic submanifold of a symplectic manifold $(M,\omega)$.
We first recall Gotay's normal form theorem for coisotropic submanifolds inside symplectic manifolds from \cite{Gotay}.
Then we show that it allows us to apply Theorem \ref{main} to recover the fact that, in the symplectic world, the coisotropic submanifolds sufficiently close to a given one are encoded by an $\li[1]$-algebra
(Theorem \ref{thm:symplectic_case}).

\begin{defi}
A 2-form   on a manifold  is {\em pre-symplectic} if it is closed and its kernel has constant rank.
\end{defi}

\begin{remark}
Let $C$ be a coisotropic submanifold of $(M,\omega)$. The pullback of $\omega$ under the inclusion $C \hookrightarrow M$ is
a pre-symplectic form which we denote by $\omega_C$.

On the other hand, starting from a pre-symplectic manifold $(C,\omega_C)$ one can construct a symplectic manifold $(\tilde{C},\tilde{\omega})$
which contains $C$ as a coisotropic submanifold in such a way that $\tilde{\omega}$ pulls back to $\omega_C$. The construction works as follows: Denote the kernel of $\omega_C$ by $E$ and its dual by $\pi: E^*\to C$. Fixing a complement $G$ of $E$ inside $TC$ yields
an inclusion $j: E^* \to T^*C$. The space $E^*$ carries a two-form
$$\Omega:= \pi^*\omega_C + j^* \omega_{T^*C}.$$
Here $\omega_{T^*C}$ denotes the canonical symplectic form on the cotangent bundle.
It is straightforward to check that $\Omega$ pulls back to $\omega_C$ and that it is symplectic on a tubular neighborhood $U$ of the zero section $C \subset E^*$. We set $(\tilde{C},\tilde{\omega})$ equal to $(U,\Omega)$ and refer to it as the {\em local model} associated to the
the pre-symplectic manifold $(C,\omega_C)$.
\end{remark}

\begin{thm}[Gotay \cite{Gotay}]\label{thm:Gotay}
Let $C$ be a coisotropic submanifold of a symplectic manifold $(M,\omega)$.

There is a symplectomorphism $\psi$ between a tubular neighborhood of $C$ inside $M$ and a tubular neighborhood of $C$ inside its local model
$(\tilde{C},\tilde{\omega})$. Moreover, the restriction of $\psi$ to $C$ is the identity.
\end{thm}

\begin{defi}\label{defi:fibrewise_homogeneous}
A differential form $\omega$ on a vector bundle $E\to C$ is called {\em fibrewise homogeneous of degree $k$}
if the following holds: given any local coordinates $(x_i)_{i=1}^m$ on $W\subset C$ and any fibre coordinates $(y_j)_{j=1}^n$
on $E|_{W}$, the differential form $\omega$
reads
$$ \sum \omega_{i_1\cdots i_r j_1\cdots j_s}(x,y)dx_{i_1}\cdots dx_{i_r}dy_{j_1}\cdots dy_{j_s}$$
on $E|_W$ with $\omega_{i_1\cdots i_r j_1\cdots j_s}(x,y)$ monomials in the fibre coordinates such that
$$\textrm{deg}(\omega_{i_1\cdots i_r j_1\cdots j_s}) + s = k.$$
In other words, the number of $y$'s and $dy$'s appearing in each summand is exactly $k$.

We denote the vector space of fibrewise homogeneous differential forms of degree $k$ on $E$ by $\Omega_{(k)}(E)$
and set $\Omega_{(\le k)}(E) := \oplus_{l\le k} \Omega_{(l)}(E)$.
\end{defi}

\begin{remark}
\hspace{0cm}
\begin{enumerate}
 \item
If a differential form $\omega$ on a vector bundle $E$ satisfies the condition of Definition \ref{defi:fibrewise_homogeneous} with respect to some
coordinate system and some choice of fibrewise linear coordinates, it also satisfies the condition with respect to any other choice of
coordinate system and fibre coordinates defined on the same open.

Hence it suffices to check the condition of Definition \ref{defi:fibrewise_homogeneous}
for an atlas $(W_{i},\phi_i)$ of $C$ and a collection of trivializations of $E|_{W_i}$.
\item The space $\Omega_{(0)}(E)$ coincides with the image of the pullback $\pi^*: \Omega(C) \to \Omega(E)$.
\item The spaces $\Omega_{(k)}(E)$ are closed under the de Rham
differential and the pullback along vector bundle maps.
\item Clearly, one can extend Definition \ref{defi:fibrewise_homogeneous} to $\Omega_{(k)}(U)$ and $\Omega_{(\le k)}(U)$
for $U$ some tubular neighborhood of $C \subset E$.
\end{enumerate}

\end{remark}

\begin{lemma}\label{lemma:symplectic_case}
Let $C$ be a coisotropic submanifold of a symplectic manifold $(M,\omega)$.

There is a diffeomorphism between a tubular neighborhood $V$ of $C$ inside $NC:=TM|_C/TC$ and a tubular neighborhood of $C$ inside $M$
such that the pullback of $\omega$ to $V$ lies in $\Omega_{(\le 1)}(V).$
\end{lemma}

\begin{proof}
Let $C$ be a coisotropic submanifold with local model $(U,\Omega)$.
By Theorem \ref{thm:Gotay}, it is enough to prove that the symplectic form $\Omega$ lies in $\Omega_{\le 1}(U)$. By definition
$$\Omega = \pi^*\omega_C + j^*\omega_{T^*C}.$$
Clearly, $\pi^*\omega_C$ lies in $\Omega_{(0)}(U)$. Further, writing $\omega_{T^*C}$ in canonical coordinates shows that it lies in
$\Omega_{(1)}(T^*C)$, hence $j^*\omega_{T^*C}\in \Omega_{(1)}(U)$.
\end{proof}

\begin{cor}\label{cor:sym}
Let $C$ be a coisotropic submanifold of a symplectic manifold $(M,\omega)$.

There is a diffeomorphism between a tubular neighborhood   $V$ of $C$ inside $NC$ and a tubular neighborhood of $C$ inside $M$
such that the pullback of the Poisson structure $\omega^{-1}$ to $V$ is a fibrewise entire Poisson structure on $NC$.
\end{cor}

\begin{proof}
Let $C$ be a coisotropic submanifold with local model $(U,\Omega)$.
By Lemma \ref{lemma:symplectic_case}, it suffices to prove that the inverse $\Omega^{-1}$ of the symplectic form
yields a fibrewise entire Poisson structure.

Since $\Omega$ lies in $\Omega_{(\le 1)}(U)$, it reads
$$ \sum_{1\le i < j \le m} \left(f_{ij}(x) + \sum_{k=1}^{{n}}y_k g^{k}_{ij}{(x)}\right) dx_i dx_j +
\sum_{\substack{1\le i\le m\\1\le k \le n}} h_{ik}(x) dx_{i} dy_k$$
on $E|_W$, where $(x_i)_{i=1}^m$ are coordinates on $W \subset C$ and $(y_{{k}})_{k=1}^n$ are fibre coordinates on $E|_W$.

We define square matrices $A$, $B_k$ of size $m+n$
by writing them as follows in block-form:
$$A= \left(\begin{array}{c|c} f(x) & h(x) \\\hline -h(x)^T & 0\end{array}\right),\;\;\;\;\;\;\;B_k=\left(\begin{array}{c|c} g^k(x) & 0 \\\hline 0 & 0\end{array}\right).$$ It is clear that the problem of determining the dependence of $\Omega^{-1}$ on the fibre coordinates reduces to the following problem:
Given an invertible  matrix $A$ and a tuple of  matrices $(B_1,\dots,B_n)$ of the same size, define
$$ M(\lambda) := A + \sum_{k=1}^n \lambda_k B_k$$
 for $\lambda=(\lambda_1,\dots,\lambda_{n})$ sufficiently close to the origin of $\RR^n$, and show that the function
$$ \lambda \mapsto M^{-1}(\lambda)$$
is entire on an open neighborhood of the origin.

This in turn holds since the general linear group   is analytic. More explicitly, using the matrix version of the geometric series $(1-x)^{-1}=\sum_{r=0}^{\infty}x^r$, one has
$$M^{-1}(\lambda)=\sum_{r=0}^{\infty}(-\sum_{k=1}^n \lambda_k A^{-1}B_k)^r A^{-1},$$
which is clearly entire near the origin.
\end{proof}

Thanks to Corollary \ref{cor:sym} we can apply Theorem \ref{main} and  recover the following result, which is -- partly in an implicit manner -- contained in \cite{OP}:

\begin{thm}\label{thm:symplectic_case}
Let $C$ be a coisotropic submanifold of a symplectic manifold $(M,\omega)$.

There is a diffeomorphism $\psi$ between a tubular neighborhood $V$ of $C$ inside $NC$ and a tubular neighborhood of $C$ inside $M$
such that for any $\alpha \in \Gamma(NC)$, with $\graph(-\alpha)$ contained in $(V,\psi^*\omega)$,
the Maurer-Cartan series $MC(\alpha)$ is convergent.

Furthermore, for any such $\alpha$ the following two statements are equivalent:
\begin{enumerate}
 \item The graph of $-\alpha$ is a coisotropic submanifold of $(V,\psi^*\omega)$.
 \item The Maurer-Cartan series $MC(\alpha)$ of $\alpha$ converges to zero.
\end{enumerate}
\end{thm}
Theorem \ref{thm:symplectic_case} asserts that in the symplectic world,
any coisotropic submanifold $C$ admits a tubular neighborhood such that the $L_{\infty}[1]$-algebra structure on $\Gamma(\wedge NC)[1]$ (see Remark \ref{remark:L_infty[1]})
controls the deformations of $C$ which are sufficiently close to $C$ with respect to the $\mathcal{C}^1$-topology.

\section{Simultaneous deformations and their obstructedness}\label{sec:sim}

Up to now we considered coisotropic submanifolds close to a given one $C$, inside a manifold with a fixed fibrewise entire Poisson structure $\pi$. Now we allow the   Poisson structure to vary inside the class of fibrewise entire Poisson structures. 
{We show that pairs $(\pi',C')$, consisting of a fibrewise entire Poisson structure $\pi'$ ``close'' to $\pi$  and a submanifold $C'$ close to $C$ and coisotropic with respect to $\pi'$, are also encoded by an $\li[1]$-algebra.}
 Then we show that the  problem of deforming simultaneously $C$ and $\pi$ is formally obstructed. In particular, we show that there is a first order deformation of $C$ and $\pi$ which can not be extended to a (smooth, or even formal) one-parameter family of deformations.

\subsection{Simultaneous deformations}
Thanks to Theorem \ref{main}, we can improve \cite[Corollary 5.3]{YaelZ}, extending it from polynomial to fibrewise entire Poisson structures.

\begin{cor}\label{cor:coiso}
Let $E\to C$ be a vector bundle and $U$ a fixed tubular neighborhood of the zero section.

There exists an $L_{\infty}[1]$-algebra structure on $\chi^{\bullet}_{\omega}(U)[2]\oplus \Gamma(\wedge E)[1]$ with the following property:
for all ${\pi}\in \chi_{\omega}^2(E)$
and $\alpha\in \Gamma(E)$ such that $\graph(-\alpha)\subset U$,
\begin{align*}
&\begin{cases}
 {\pi} \text{ is a Poisson structure   } \\
\graph({-\alpha}) \text{ is a coisotropic submanifold of  } (U, \pi)
\end{cases}\\
\Leftrightarrow \quad
&
( {\pi}[2],\alpha[1]) \text{ is a Maurer-Cartan element } \text{ of   } \chi^{\bullet}_{\omega}(U)[2]\oplus \Gamma(\wedge E)[1].
\end{align*}
Explicitly, the $L_{\infty}[1]$-algebra structure is given by the following multibrackets
 (all others vanish):
\begin{align*}
\lambda_1(X[1])&= P(X),\\
\lambda_2(X[1],Y[1])&=(-1)^{|X|}[X,Y][1],\\
\lambda_{n+1}(X[1],a_1,\dots,a_n)&=P([\dots[X,a_1],\dots,a_n])\;\;\;\;\;\; \text{for all }n\ge 1,
\end{align*}
where  $X,Y\in \chi^{\bullet}_{\omega}(U)[1]$, $a_1,\dots,a_n \in \Gamma(\wedge E)[1]$, and $[\cdot,\cdot]$ denotes the Schouten bracket on $\chi^{\bullet}_{\omega} (U)[1]$.
 \end{cor}
\begin{proof}
Use Theorem \ref{main} in order to apply \cite[Corollary 1.13]{YaelZ}.
\end{proof}

\subsection{Obstructedness}
We review what it means for a deformation problem to be formally obstructed in an abstract setting:
\begin{itemize}
\item Let $W_0$ be a topological vector space and $G$ a subset (so the elements of $G$ are distinguished among elements of $W_0$). Let $x\in G$. A \emph{deformation} of $x$  is just an element of $G$ (usually though of being ``nearby'' $x$ in some sense).
\item We say that  an $\li[1]$-algebra $(W,\{\lambda_k\}_{k\ge 1})$, whose degree zero component is exactly the above vector space, \emph{governs}
 the  deformation problem of $x$ if the following is satisfied:  $y\in W_0$ satisfies\footnote{By definition, this means that the Maurer-Cartan series $MC(y)$ converges to zero.}  the Maurer-Cartan equation if{f} $x+y$ is a deformation of $x$.
\end{itemize}

\begin{defi} Suppose  a certain  deformation problem is governed by the $\li[1]$-algebra $(W,\{\lambda_k\}_{k\ge 1})$.
The deformation problem is said to be \emph{formally unobstructed} if for any class $A \in H^1(W)$ there is a sequence $\{z_k\}_{k\ge 1}\subset W_0$, with $z_1$ being $\lambda_1$-closed and representing the class $A$, so that $\sum_{k\ge 1}z_kt^k$
is a (formal) solution of the Maurer-Cartan equation for $(W,\{\lambda_k\}_{k\ge 1})$. Here  $H(W)$ is the cohomology of the complex  $(W,\lambda_1)$, and $t$ is a formal variable.
\end{defi}

Later on, to show the formal {obstructedness},
we will use the criteria   \cite[Theorem 11.4]{OP},  which reads as follows:
\begin{thm}\label{thm:opobstr}
Suppose  a certain  deformation problem   is governed by  the $\li[1]$-algebra $(W,\{\lambda_k\}_{k\ge 1})$.
Define the Kuranishi map
\begin{equation}
\Kr \colon H^1(W)  \to H^2(W), \;\;\ [z]\mapsto [\lambda_2(z,z)].
\end{equation}
If $A\in H^1(W)$ satisfies $\Kr(A)\neq 0$, then there is no formal solution  $\sum_{k\ge 1}z_kt^k$
of the Maurer-Cartan equation with $[z_1]=A$.
In particular, if the Kuranish map is not identically zero, the deformation problem is formally obstructed.
\end{thm}

\subsection{Obstructedness of the extended deformation problem}

 Given a symplectic manifold $(M,\omega)$ and a coisotropic submanifold $C$ of $(M,\omega)$, Oh-Park \cite[\S 11]{OP} showed that the problem of deforming $C$ to nearby coisotropic submanifolds of $(M,\omega)$ is formally obstructed. Oh-Park ask whether the deformation problem remains formally obstructed if one  allows \emph{both} the symplectic form $\omega$ \emph{and} the submanifold $C$ to
 vary (they refer to this as ``extended deformation problem'' \cite[\S 13, p. 355]{OP}). In this subsection we answer this question in the setting of fibrewise entire Poisson structures, showing that the extended deformation problem is formally obstructed too, see Corollary \ref{cor:ob}.

Let $E\to C$ be a vector bundle and $\pi$ a fibrewise entire Poisson structure which is defined on a tubular neighborhood $U=\dom(\pi)$ of $C$ in $E$. Suppose that $C$ is coisotropic with respect to $\pi$. These conditions are equivalent to\footnote{Here we view $\pi$ as an element of
$\chi^{\bullet}_{\omega}(U)[1]$ (notice the degree shift).}
 $(\pi[1],0)$ satisfying the Maurer-Cartan equation in the $\li[1]$-algebra of Corollary \ref{cor:coiso}.

The $\li[1]$-algebra governing the deformations of  the coisotropic submanifold $C$ and the Poisson structure $\pi$ is obtained from the $\li[1]$-algebra of Corollary \ref{cor:coiso}, by twisting it by the Maurer-Cartan element $(\pi[1],0)$ (\cite[Proposition 4.4]{GetzlerAnnals}, see also \cite[\S 1.3]{YaelZ}). Hence it is
$\chi^{\bullet}_{\omega}(U)[2]\oplus \Gamma(\wedge E)[1]$, with multi-brackets:
\begin{align*}
\lambda_1(X[1])&= (-[\pi,X][1],P(X)),\\
\lambda_1(a)&= (0,P([\pi,a])),\\
\lambda_2(X[1],Y[1])&=(-1)^{|X|}[X,Y][1],\\
\lambda_{n}(a_1,\dots,a_n)&=P([\dots[\pi,a_1],\dots,a_n])\;\;\;\;\;\; \text{for all }n\ge 1,\\
\lambda_{n+1}(X[1],a_1,\dots,a_n)&=P([\dots[X,a_1],\dots,a_n])\;\;\;\;\;\; \text{for all }n\ge 1,
\end{align*}
where  $X,Y\in \chi^{\bullet}_{\omega}(U)[1]$, $a_1,\dots,a_n \in \Gamma(\wedge E)[1]$.
We denote the above $\li[1]$-algebra by $W(C,\pi)$.\\

\begin{defi}
Let $\pi$ be a fibrewise entire Poisson structure on $E\to C$. Suppose that $C$ is coisotropic with
respect to $\pi$.
The {\em extended deformation problem} of $C$ is the deformation problem governed by
the $L_{\infty}[1]$-algebra structure $W(C,\pi)$.
\end{defi}

We are now ready to display an example showing that, in general, the extended deformation problem (with the requirement that the Poisson structures involved be fibrewise entire)
is formally obstructed. The example is exactly the one previously used by the second author \cite{coisoemb} and by Oh-Park \cite[Ex. 11.4]{OP}.

\begin{prop}\label{prop:T4}
Let $C=\RR^4/\ZZ^4$ be  the 4-dimensional torus with coordinates $(y_1,y_2,q_1,q_2)$. Set $E=\RR^2\times C$, and denote by $p_1,p_2$ the coordinates on $\RR^2$. Consider the Poisson structure $\pi$ on $E$ obtained by inverting the symplectic form
$$\Omega:=dy_1dy_2+(dq_1dp_1+dq_2dp_2).$$

There exists $a\in \Gamma(E)[1]$ such that the
class
 $A\in H^1(W(C,\pi))$ represented by
$(0[1],a)$ satisfies $\Kr(A)\neq 0$.
\end{prop}

\begin{cor}\label{cor:ob}
The extended deformation problem for the coisotropic submanifold $C$ in the symplectic manifold $(E,\Omega)$ as in Proposition \ref{prop:T4}  is formally obstructed.
\end{cor}

\begin{proof}[Proof of Proposition \ref{prop:T4}]

Recall that the Poisson structure $\pi$ on $E$ makes $T^*E$ into a Lie algebroid over $E$ with anchor map
$\sharp$ given by $\sharp(\gamma):=\pi(\gamma,\cdot)$. Because $\pi$ is symplectic, the anchor $\sharp$
is an isomorphism of Lie algebroids from $T^*E$ to $TE$.
Dually, we obtain a vector bundle isomorphism
$$\sharp^* \colon \wedge T^*E \to \wedge TE$$
 which intertwines the Lie algebroid differentials
$d_{dR}$ and $[\pi,\cdot]$.

Since $C$ is coisotropic with respect to $\pi$, its conormal bundle $(TC)^{\circ} \cong E^*$ is a Lie subalgebroid of $T^*E$.
The Lie algebroid isomorphism $\sharp$ restricts to an isomorphism of Lie subalgebroids $\tilde{\sharp} \colon E^* \to F$, where $F=\textrm{span}\{\pd{q_1},\pd{q_2}\}$ is the kernel of the pullback of $\Omega$ to $C$.
By dualizing we obtain
$$\tilde{\sharp}^* \colon \Gamma(\wedge F^*)\to  \Gamma(\wedge E)$$ which intertwines the Lie algebroid differentials
$d_{F}=d_{dR}|_{\Gamma(\wedge F^*)}$ and $P([\pi,\cdot])$.
Notice that 
\begin{equation}\label{eq:sh}
P\circ{\sharp}^*=\tilde{\sharp}^*\circ \iota^*\colon \Gamma(\wedge T^*E)\to \Gamma(E),
\end{equation}
 where $\iota\colon F\to TE$ is the inclusion.

Now let $a\in \Gamma(E)[1]$ such that $(0[1],a)$ is a $\lambda_1$-closed element of $W(C,\pi)$, i.e. such that $P([\pi,a])=0$. Suppose that $A=[(0[1],a)]$ satisfies $\Kr(A)=0$. In other words, assume that
there are $\tau\in \chi^{2}_{\omega}(U)[1]$ and $b\in  \Gamma(E)[1]$
such that $$\lambda_1(\tau[1],b)=(-[\pi,\tau][1],P(\tau+[\pi,b])) \;\;\;
\overset{!}{=}\;\;\;
\lambda_2((0[1],a),(0[1],a))=(0[1],P([[\pi,a],a])),$$ holds, i.e. so that
\begin{eqnarray}
\label{eq1}[\pi,\tau]&=&0, \qquad \qquad \textrm{and}\\
\label{eq2}P(\tau+[\pi,b])&=&P([[\pi,a],a])
\end{eqnarray}
are satisfied.

Let $$\beta:=(\tilde{\sharp}^*)^{-1}(P[[\pi,a],a])\in \Gamma(\wedge^2F^*)$$
and consider the submanifolds
$\Sigma_y=\{(y,q): q\in \RR^2/\ZZ^2\}$ and $\Sigma_{y'}$ of $C$, corresponding to two
fixed points $y$ and $y'$ of $\RR^2/\ZZ^2$. Let
$\Delta_{y,y'}:=\{(y+t(y'-y),q): t\in [0,1], q \in \RR^2/\ZZ^2\}$, a 3-dimensional submanifold of $C$ with boundary $
\Sigma_y\cup\overline{\Sigma_{y'}}$, where the bar indicates orientation reversal.
Using equation \eqref{eq2} to rewrite $\beta$, we see that $\beta$
 satisfies
\begin{equation*}
\int_{\Sigma_y\cup \overline{\Sigma_{y'}}}\beta=
\int_{\Sigma_y\cup \overline{\Sigma_{y'}}}(\tilde{\sharp}^*)^{-1}P\tau
=\int_{\Sigma_y\cup \overline{\Sigma_{y'}}}\iota^*(({\sharp}^*)^{-1}\tau)
=\int_{\Delta_{y,y'}}d_{dR}(({\sharp}^*)^{-1}\tau)=0,
\end{equation*}
where
 \begin{itemize}
\item in the first equation we used that
$(\tilde{\sharp}^*)^{-1}$ maps $P([\pi,b])$  to $d_F((\tilde{\sharp}^*)^{-1} b)$,
 and
Stokes' theorem on $\Sigma_y\cup \overline{\Sigma_{y'}}$,
\item in the  second   we used eq. \eqref{eq:sh},
\item in the third   we used Stokes' theorem on $\Delta_{y,y'}$,
\item in the  fourth  we used that
$(\sharp^*)^{-1}$ maps $\tau$ to a  $d_{dR}$-closed form (a consequence of equation \eqref{eq1}).
\end{itemize}
We conclude that the function $F\colon \RR^2/\ZZ^2\to \RR, y \mapsto \int_{\Sigma_y}\beta$ induced by the section $a$ is a constant function.

Now we make a specific choice for $a\in \Gamma(E)[1]$, namely we choose it to be given by
\begin{equation}\label{eq:sin}
(a_1,a_2) \colon C \to \RR^2, \quad
(y_1,y_2,q_1,q_2)\mapsto (\sin(2\pi y_1),\sin(2\pi y_2)).
\end{equation}
The condition $P([\pi,a])=0$ is equivalent to $(\tilde{\sharp}^*)^{-1} a$ being $d_F$-closed, which
is satisfied, as
 $(\tilde{\sharp}^*)^{-1} a=-\sin(2\pi y_1)dq_1-\sin(2\pi y_2)dq_2$.

Now, a direct computation
shows $$P([[\pi,a],a])=2\left(\frac{\partial a_1}{\partial y_1}\frac{\partial a_2}{\partial y_2}-\frac{\partial a_1}{\partial y_2}\frac{\partial a_2}{\partial y_1}\right)\partial_{p_1}\wedge\partial_{p_2}=
8\pi^2 \cos(2\pi y_1)\cos(2\pi y_2)\partial_{p_1}\wedge\partial_{p_2}.$$
The function $F$ is therefore given by  $$F(y)= \int_{\Sigma_y}\beta =8\pi^2\cos(2\pi y_1)\cos(2\pi y_2),$$
so in particular it is not a constant function
and
we deduce that the specific choice of $a$ as in equation \eqref{eq:sin} satisfies $\Kr(A)\neq 0$.
\end{proof}

\bibliographystyle{habbrv}
\bibliography{Actbib}
\end{document}